\documentclass[12pt]{amsart}

\usepackage{amssymb}
\usepackage{bm}
\usepackage{graphicx}
\usepackage{psfrag}
\usepackage{color}
\usepackage{url}
\usepackage{enumerate}

\newcommand{\excise}[1]{}%{$\star$\textsc{#1}$\star$}

\makeatletter
\newtheorem*{rep@theorem}{\rep@title}
\newcommand{\newreptheorem}[2]{%
\newenvironment{rep#1}[1]{%
 \def\rep@title{#2 \ref{##1}}%
 \begin{rep@theorem}}%
 {\end{rep@theorem}}}
\makeatother

\newtheorem{thm}{Theorem}[section]
\newtheorem{lemma}[thm]{Lemma}

\newtheorem{cor}[thm]{Corollary}
\newtheorem{prop}[thm]{Proposition}
\newtheorem{conj}[thm]{Conjecture}

\theoremstyle{definition}
\newtheorem{example}[thm]{Example}
\newtheorem{remark}[thm]{Remark}
\newtheorem{defn}[thm]{Definition}

\numberwithin{equation}{section}

\newreptheorem{cor}{Corollary}

\newcommand{\ring}[1]{\ensuremath{\mathbb{#1}}}

\renewcommand\>{\rangle}
\newcommand\<{\langle}

\newcommand\NN{\ring{N}}

\newcommand\QQ{\ring{Q}}

\newcommand\ZZ{\ring{Z}}

\renewcommand\aa{{\vec a}}

\DeclareMathOperator\image{Im} % Hom
\begin{document}%%%%%%%%%%%%%%%%%%%%%%%%%%%%%%%%%%%%%%%%%%%%%%%%%%%%%%%%
%%%%%%%%%%%%%%%%%%%%%%%%%%%%%%%%%%%%%%%%%%%%%%%%%%%%%%%%%%%%%%%%%%%%%%%%

\mbox{}
%\vspace{-2ex}%-1.1743pt}
\title{On the set of elasticities in numerical monoids\qquad}
\author{Thomas Barron}
\address{Mathematics Department\\University of Kentucky\\Lexington, KY 40506}
\email{thomas.barron@uky.edu}
\author{Christopher O'Neill}
\address{Mathematics Department\\Texas A\&M University\\College Station, TX 77843}
\email{coneill@math.tamu.edu}
\author{Roberto Pelayo}
\address{Mathematics Department\\University of Hawai`i at Hilo\\Hilo, HI 96720}
\email{robertop@hawaii.edu}

\date{\today}

\begin{abstract}
\hspace{-2.05032pt}
In an atomic, cancellative, commutative monoid $S$, the elasticity of an element provides a coarse measure of its non-unique factorizations by comparing the largest and smallest values in its set of factorization lengths (called its length set).  In this paper, we show that the set of length sets $\mathcal L(S)$ for any arithmetical numerical monoid $S$ can be completely recovered from its set of elasticities $R(S)$; therefore, $R(S)$ is as strong a factorization invariant as $\mathcal L(S)$ in this setting.   For general numerical monoids, we describe the set of elasticities as a specific collection of monotone increasing sequences with a common limit point of $\max R(S)$.
\end{abstract}
\maketitle

% \setcounter{tocdepth}{1}
% \tableofcontents

\vspace{-1cm}

%%%%%%%%%%%%%%%%%%%%%%%%%%%%%%%%%%%%%%%%%%%%%%%%%%%%%%%%%%%%%%%%%%%%%%%%%
\section{Introduction}\label{s:intro}%%%%%%%%%%%%%%%%%%%%%%%%%%%%%%%%%%%%
%raggedbottom%%%%%%%%%%%%%%%%%%%%%%%%%%%%%%%%%%%%%%%%%%%%%%%%%%%%%%%%%%%%

In studying the non-unique factorization theory of atomic monoids, the development of several invariants -- such as delta sets~\cite{delta} and $\omega$-primality~\cite{quasi} -- has provided significant insight.  Of particular interest is the set of length sets $\mathcal L(S)$ for an atomic monoid $S$, which has as its elements the sets of factorization lengths of elements in $S$~\cite{setoflengthsets,geroldingerlengthsets,realizthm}.
The following longstanding conjecture states that, with one exception, the set of length sets is a perfect invariant for the important class of \emph{block monoids} $\mathcal B(G)$ of zero-sum sequences over a finite Abelian group $G$ \cite[Section~7.3]{nonuniq}.  %, which play a fundamental role in non-unique factorization theory.  

\begin{conj}\label{c:introblocklensets}
Given two finite Abelian groups $G$ and $G'$ with $|G|, |G'| > 3$, we have $\mathcal L(\mathcal B(G)) = \mathcal L(\mathcal B(G'))$ implies $\mathcal B(G) = \mathcal B(G')$.  
\end{conj}

In contrast to the above conjecture, the authors of~\cite{setoflengthsets} show that two distinct numerical monoids (co-finite, additive submonoids of $\NN$) can have the same length sets.  In this paper, we investigate the elasticity $\rho(n)$ of elements $n$ in a numerical monoid $S$.  This invariant, computed as the quotient of the largest factorization length by the smallest, provides a coarse measure of an element's non-unique factorizations.  We now state our main result concerning the set $R(S) = \{\rho(n) : n \in S\}$ of elasticities of $S$.  

\begin{thm}\label{t:introelastlencomp}
For distinct arithmetical numerical monoids $S = \<a, a + d, \ldots, a + kd\>$ and $S' = \<a', a' + d', \ldots, a' + k'd'\>$, the following are equivalent: 
\begin{enumerate}
\item $R(S) = R(S')$.  
\item $\mathcal L(S) = \mathcal L(S')$.  
\end{enumerate}
\end{thm}

Therefore, for the class of arithmetical numerical monoids (numerical monoids generated by an arithmetic sequence), the set of elasticities is as strong an invariant as the set of length sets.  In contrast, we also provide Example~\ref{e:elastlencomp}, which gives two non-arithmetical numerical monoids with identical sets of elasticities, but distinct sets of length sets.

After developing our main result in Section~\ref{s:arithmetical}, we provide a full characterization of the set of elasticities for any numerical monoid, thereby completing a coarser description provided by Chapman, Holden, and Moore~\cite{elasticity}.  This characterization (Corollary~\ref{c:elasticityset}) demonstrates the stark contrast between the set of length sets, which is often very large and hard to compute, with the set of elasticities, which we describe as a union of monotonically increasing sequences with a common limit point of $\max R(S)$.  For arithmetical numerical monoids, this characterization of $R(S)$ takes the form of a complete parametrization~(Theorem~\ref{t:arithelastset}).

%%%%%%%%%%%%%%%%%%%%%%%%%%%%%%%%%%%%%%%%%%%%%%%%%%%%%%%%%%%%%%%%%%%%%%%%%
\section{Background}\label{s:background}%%%%%%%%%%%%%%%%%%%%%%%%%%%%%%%%%
%raggedbottom%%%%%%%%%%%%%%%%%%%%%%%%%%%%%%%%%%%%%%%%%%%%%%%%%%%%%%%%%%%%

In this section, we provide definitions and previous results related to the elasticity of elements in a numerical monoid.  In what follows, let $\NN$ denote the set of non-negative integers.  Unless otherwise stated, we will assume that $S$ has minimal generating set $\{g_1, \ldots, g_k\}$ with $g_1 < \cdots < g_k$ and $\gcd(g_1, \ldots, g_k) = 1$.    

\begin{defn}\label{d:factorization}
Let $S = \<g_1, \ldots, g_k\>$ be a numerical monoid with minimal generating set $\{g_1, \ldots, g_k\}$, and fix $n \in S$.  An element $\aa = (a_1,\ldots, a_k) \in \NN^k$ is a \emph{factorization} of $n$ if $n = a_1 g_1 + \cdots + a_k g_k$, and its \emph{factorization set} is given by 
$$\mathsf Z(n) = \left\{(a_1, \ldots, a_k) \in \NN^k \, : \, a_1 g_1 + \cdots + a_k g_k = n\right\}.$$
The \emph{length} of the factorization $\aa$, denoted $|\aa|$, is given by $a_1 + \cdots + a_k$.  For each $n$, the \emph{length set} of $n$ is the set $\mathsf L(n) = \left\{|a| \, : \vec a \in \mathsf Z(n)\right\}$, and the \emph{set of length sets} of the monoid $S$ is given by $\mathcal L(S) = \left\{\mathsf L(n) \, : \, n \in S\right\}.$
\end{defn}

\begin{remark}\label{r:length}
While the length set of an element in a numerical monoid is a helpful measure of its non-unique factorizations, some information is lost when passing from $\mathsf Z(n)$ to $\mathsf L(n)$.  For example, in $S = \<3,5,7\>$, the element $10 \in S$ has as its two distinct factorizations $(1,0,1)$ and $(0,2,0)$, both of which have length $2$.  Thus, even though $\mathsf L(10) = \{2\}$ is singleton, the element $10$ has multiple factorizations.  This phenomenon is common in numerical monoids, especially those minimally generated by arithmetic sequences of length $3$ or greater.  See Section~\ref{s:arithmetical} for a more detailed analysis of such~monoids.
\end{remark}

In a numerical monoid, length sets of elements are finite.  Thus, analyzing the relationship between an element's maximal and minimal lengths provides a meaningful, albeit coarse, gauge of the non-uniqueness of its factorizations.  This concept, known as the elasticity of an element, is defined below.

\begin{defn}\label{d:elasticity}
For an element $n \in S$ of a numerical monoid, we denote by 
$$M_S(n)=\max \mathsf L(n) \,\,\,\, \text{and} \,\,\,\, m_S(n)=\min \mathsf L(n)$$
the \emph{maximal} and \emph{minimal length} of $n$, respectively.  The ratio 
$$\rho_S(n) = M_S(n)/m_S(n)$$
is called the \emph{elasticity} of $n$.  When there is no ambiguity, we omit the subscripts and simply write $M(n)$, $m(n)$, and $\rho(n)$.  The \emph{set of elasticities} of $S$ is given by 
$$R(S) = \left\{\rho(n) \, : \, n \in S\right\},$$
and the \emph{elasticity of $S$} is given by the supremum of this set: $\rho(S) = \sup R(S)$.
\end{defn}

\begin{defn}\label{d:arithmetical}
A numerical monoid $S$ is \emph{arithmetical} if it is minimally generated by an arithmetic sequence of positive integers, that is, 
$$S = \<a, a+d, \ldots, a+kd\>$$
for positive integers $a, d$, and $k$.  Unless otherwise stated, when the generating set of a numerical monoid is expressed in the form $a, a + d, \ldots, a + kd$, it is assumed that $\gcd(a,d) = 1$ and $1 \le k < a$.  
\end{defn}

We conclude this section by recalling some relevant results from the literature.  Theorem~\ref{t:holdenmoore} provides some coarse properties of the set of elasticites of a numerical monoid.  Proposition~\ref{p:arithminmax} is a consequence of \cite[Theorem~2.2]{setoflengthsets}, and characterizes the functions $M_S$ and $m_S$ for any arithmetical numerical monoid $S$.  Lastly, Theorem~\ref{t:setoflengthsets} appeared as \cite[Theorem~3.2]{setoflengthsets} and is vital to the proof of Corollary~\ref{c:elastlencomp}.  

\begin{thm}[{\cite[Theorem 2.1 \& Corollary 2.3]{elasticity}}]\label{t:holdenmoore}
If $S$ is a numerical monoid minimally generated by $g_1 < \cdots < g_k$, then $\rho(S) = g_k/g_1$ is the unique accumulation point of $R(S)$, and there exists an $n \in S$ such that $\rho(n) = \rho(S)$.  
\end{thm}

\begin{prop}\label{p:arithminmax}
Fix an arithmetical numerical monoid $S = \<a, a + d, \ldots, a + kd\>$ with $\gcd(a,d) = 1$ and $k < a$.  For $n \in S$, we have the following.
\begin{enumerate}[(a)]
\item If $n = x(a + kd) - yd$ for $0 \le y < a + kd$, then $m(n) = x$.  
\item If $n = x'a + y'd$ for $0 \le y' < a$, then $M(n) = x'$.  
\end{enumerate}
\end{prop}

\begin{thm}[{\cite[Theorem~3.2]{setoflengthsets}}]\label{t:setoflengthsets}
Fix two distinct numerical monoids $S = \<a, a + d, \ldots, a + kd\>$ and $S' = \<a', a' + d', \ldots, a' + k'd'\>$ 
for $\gcd(a,d) = \gcd(a',d') = 1$, $1 \le k < a$ and $1 \le k < a'$.  The following statements are equivalent: 
\begin{enumerate}[(a)]
\item $\mathsf L(S) = \mathsf L(S')$, and 
\item $d = d'$, $\frac{a}{k} = \frac{a'}{k'}$, $\gcd(a,k) \ge 2$ and $\gcd(a',k') \ge 2$.  
\end{enumerate}
\end{thm}

%%%%%%%%%%%%%%%%%%%%%%%%%%%%%%%%%%%%%%%%%%%%%%%%%%%%%%%%%%%%%%%%%%%%%%%%%%%%%%%%%%%%%%%%%%%%%%%%%%%%%%%%
\section{Elasticity sets for arithmetical numerical monoids}\label{s:arithmetical}%%%%%%%%%%%%%%%%%%%%%%
%raggedbottom%%%%%%%%%%%%%%%%%%%%%%%%%%%%%%%%%%%%%%%%%%%%%%%%%%%%%%%%%%%%%%%%%%%%%%%%%%%%%%%%%%%%%%%%%%%

Remark~\ref{r:length} demonstrates that information is lost when passing from $\mathsf Z(n)$ to $\mathsf L(n)$.  Since only the ratio of $\max \mathsf L(n)$ and $\min \mathsf L(n)$ is retained when passing from $\mathsf L(n)$ to $\rho(n)$, one might expect that further information is lost when passing from the set of length sets $\mathcal L(S)$ to the set of elasticities $R(S)$.  While this is true in general (see Example~\ref{e:elastlencomp}), when $S$ is an arithmetical numerical monoid, $\mathcal L(S)$ can be recovered from $R(S)$.  This is the content of Corollary~\ref{c:elastlencomp}, the main result of this section.  

For an arithmetical numerical monoid $S = \<a, a + d, \ldots, a + kd\>$, Theorem~\ref{t:setoflengthsets} states that the values $d$ and $a/k$ can both be recovered from $\mathcal L(S)$, and that if $\gcd(a,k)  = 1$, then $\mathcal L(S)$ cannot coincide with $\mathcal L(S')$ for any arithmetical numerical monoid $S' \ne S$.  In order to prove Corollary~\ref{c:elastlencomp} we show that each of these results also holds true for the set of elasticities $R(S)$.  

The proof of Corollary~\ref{c:elastlencomp} comes in two steps.  First, Proposition~\ref{p:arithstep} proves that $d$ can be recovered from $R(S)$.  This also implies the value of $a/k$ can be recovered; see Remark~\ref{r:arithstep}.  Second, Theorem~\ref{t:elastlencomp} ensures that if $\gcd(a,k) = 1$, then $R(S)$ does not coincide with $R(S')$ for any arithmetical numerical monoid $S'$.  

\begin{example}\label{e:arithelast}
Figure~\ref{f:arith_elast_plots} plots the elasticities of elements of $S = \<7,12,17,22\>$.  Notice that the graph appears to be a collection of ``slices'', each consisting of several ``rows'' of points with the same elasticity value.  Theorem~\ref{t:arithelastset} uses Proposition~\ref{p:arithminmax} to eliminate much of the redundancy in $R(S)$ by reparametrizing in terms of these slices and rows (Definition~\ref{d:elasttuple}), thereby simplifying many computations in results throughout this section.  See Example~\ref{e:arithelastparams} for a description of these values.  
\end{example}

\begin{figure}
\begin{center}
\includegraphics[width=5.9in]{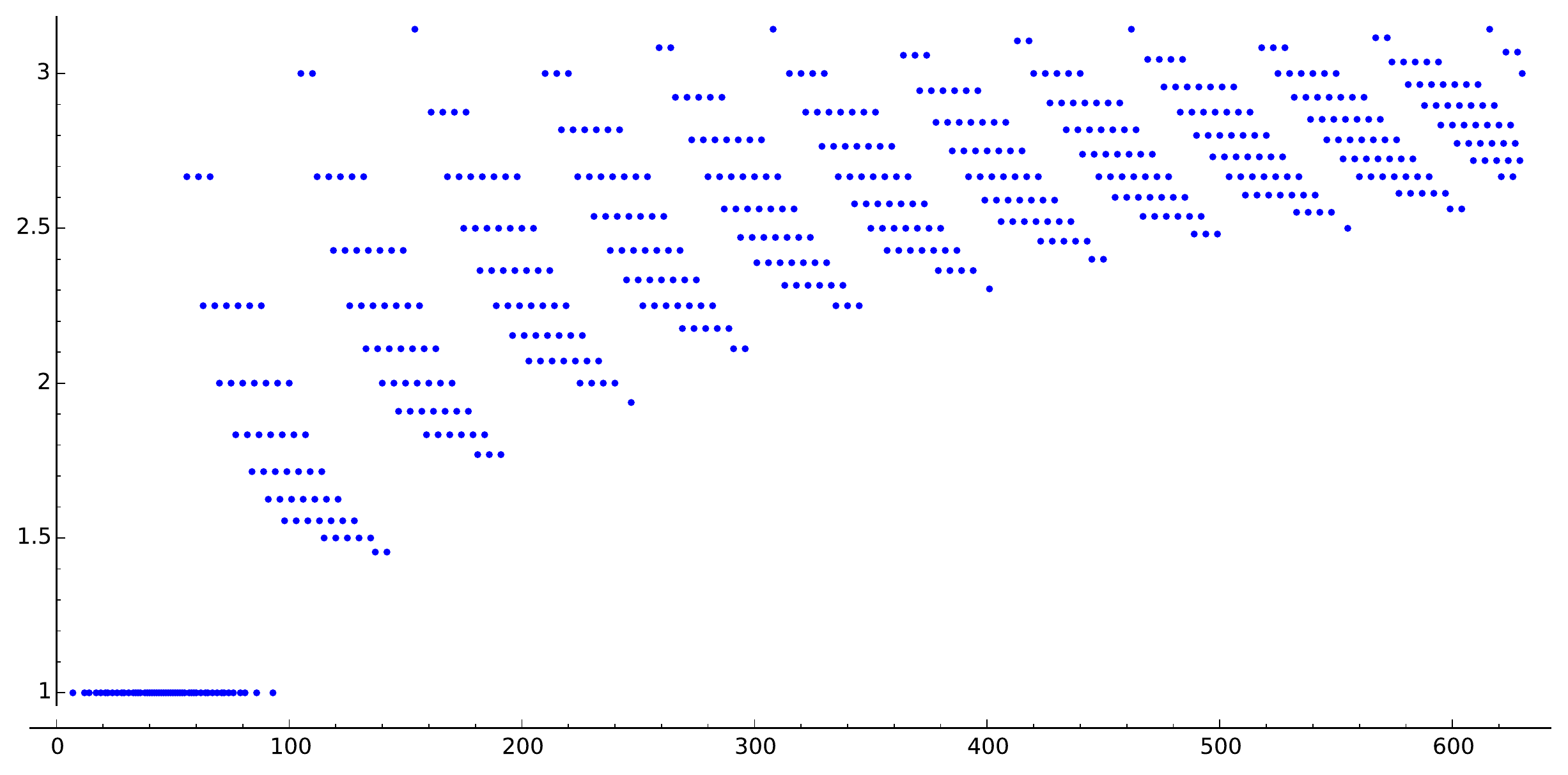}
% \hspace{0.2in}
% \includegraphics[width=2.9in]{14_6_3.pdf}
\end{center}
\caption{Plot of elasticities of $S = \<7,12,17,22\>$.}
\label{f:arith_elast_plots}
\end{figure}

\begin{defn}\label{d:elasttuple}
Fix an arithmetical numerical monoid $S = \<a, a + d, \ldots, a + kd\>$ with $\gcd(a,d) = 1$ and $k < a$.  An element $(c,s,x) \in \ZZ^3$ is an \emph{$S$-elasticity tuple} if $c \ge 0$, $0 \le s < k$, and 
$$\left\lceil \frac{sa}{k} \right\rceil \le x \le \left\lfloor \frac{sa + 2(a - 1)}{k} \right\rfloor + d.$$
The value $ck + s$ is the \emph{slice} of $(c,s,x)$, and $x$ is called the \emph{row} of $(c,s,x)$.  An $S$-elasticity tuple $(c,s,x)$ is \emph{minimal} if $x = \left\lceil \frac{sa}{k} \right\rceil$ and \emph{maximal} if $x = \left\lfloor \frac{sa + 2(a - 1)}{k} \right\rfloor + d$.  Write $\mathcal E(S)$ for the set of $S$-elasticity tuples, and define $\rho_S:\mathcal E(S) \to \QQ$ as 
$$\rho_S(c,s,x) = \frac{c(a + kd) + x + sd}{ca + x}.$$
\end{defn}

\begin{example}\label{e:arithelastparams}
For arithmetical $S = \<a, a + d, \ldots, a + kd\>$, each $S$-elasticity tuple $(c,s,x)$ corresponds to the elasticity $\rho_S(c,s,x)$ occuring in the $(ck + s)$-th slice (where every elasticity in the 0-th slice is 1).  Minimal $S$-elasticity tuples (those with a minimal $x$ value for their slice) correspond to the largest elasticity in the slice, and each successive value of $x$ corresponds to the next row down in the slice.  Maximal $S$-elasticity tuples play a key role in Lemma~\ref{l:extraelast} and Theorem~\ref{t:elastlencomp}; see Example~\ref{e:extraelast}.  

Since the tuple $(c,s,x)$ corresponds to the $(ck + s)$-th slice, it is tempting to use the ordered pair $(ck + s,x)$ in place of $(c,s,x)$ in Definition~\ref{d:elasttuple}.  However, the individual values of $c$ and $s$ are used in nearly every proof in this section.  In particular, the slices whose $S$-elasticity tuples have $s = 0$ are precisely those whose highest elasticity value is $\sup R(S)$.  Indeed, the arithmetical numerical monoid $S$ depicted in Figure~\ref{f:arith_elast_plots} has $k = 3$, and every third slice has $\rho(S)$ as its highest value.  
\end{example}

We now state Theorem~\ref{t:arithelastset}, which ensures that the parametrization given in Definition~\ref{d:elasttuple} produces the correct elasticity set.  

\begin{thm}\label{t:arithelastset}
$R(S) = \rho_S(\mathcal E(S))$ for any arithmetical $S = \<a, a + d, \ldots, a + kd\>$.  
\end{thm}

\begin{proof}
We begin by showing that for each $n \in S$, the elasticity $\rho(n) = \rho(c,s,x)$ for some $(c,s,x) \in \mathcal E(S)$.  First, write $n = x'a + y'd = x''(a + kd) - y''d$ for $x', x'', y', y'' \ge 0$, $y' < a$, and $y'' < a + kd$.  By Proposition~\ref{p:arithminmax}, $M(n) = x'$ and $m(n) = x''$, and since $x', x'' \in \mathsf L(n)$, we have $d \mid x' - x''$ by~\cite[Theorem~3.9]{delta}.  Fix $c \ge 0$ and $0 \le s < k$ such that $x' - x'' = (ck + s)d$, and let $x = x'' - ca = x' - c(a + kd) - sd$.  Notice that 
$$\begin{array}{rcl}
x(a + kd) - y''d
&=& (x'' - ca)(a + kd) - y''d = n - ca(a + kd) \\
&=& (x' - c(a + kd))a + y'd = (x + sd)a + y'd
\end{array}$$
which implies that $xk = sa + y' + y''$.  Since $y' + y'' \le 2a + kd - 2$, this means $sa \le xk \le sa + 2a + kd - 2$, which yields 
$$\left\lceil \frac{sa}{k} \right\rceil \le x \le \left\lfloor \frac{sa + 2a + kd - 2}{k} \right\rfloor = \left\lfloor \frac{sa + 2(a - 1)}{k} \right\rfloor + d.$$
This means $(c,s,x) \in \mathcal E(S)$ and
$$\rho_S(c,s,x) = \frac{c(a + kd) + sd + x}{ca + x} = \frac{x'}{x''} = \rho_S(n),$$
which proves $R(S) \subset \rho_S(\mathcal E(S))$.  

Conversely, fix $(c,s,x) \in \mathcal E(S)$.  The assumptions on $x$ ensure that 
$$sa \le xk \le sa + 2(a - 1) + kd$$
meaning $0 \le xk - sa \le a + (a + kd) - 2$.  Fix $y', y'' \ge 0$ such that $y' < a$, $y'' < a + kd$, and $y' + y'' = xk - sa$.  Choosing 
$$n = (c(a + kd) + x + sd)a + y'd = (ca + x)(a + kd) - y''d \in S$$
yields $\rho_S(n) = \rho_S(c,s,x)$, meaning $\rho_S(\mathcal E(S)) \subset R(S)$.  
\end{proof}

In the terminology of Example~\ref{e:arithelastparams}, Lemma~\ref{l:elastcomp} states that increasing an $S$-elasticity tuple's slice produces larger elasticities, and increasing its row yields smaller elasticities.  

\begin{lemma}\label{l:elastcomp}
Fix an arithmetical numerical monoid $S = \<a, a + d, \ldots, a + kd\>$.  
\begin{enumerate}[(a)]
\item Any $(c,s,x), (c',s',x) \in \mathcal E(S)$ with $ck + s \le c'k + s'$ satisfy $\rho_S(c,s,x) \le \rho_S(c',s',x)$.  
\item Any $(c,s,x), (c,s,x') \in \mathcal E(S)$ with $x \le x'$ satisfy $\rho_S(c,s,x') \le \rho_S(c,s,x)$.  
\end{enumerate}
\end{lemma}

\begin{proof}
The claim follows directly upon comparing fractions and observing that 
$$1 \le \rho_S(c,s,x) = \frac{c(a + kd) + x + sd}{ca + x} \le \frac{a + kd}{a}$$
for every $(c,s,x) \in \mathcal E(S)$.  
\end{proof}

We now use the parametrization of $R(S)$ provided by Theorem~\ref{t:arithelastset} to prove Corollary~\ref{c:elastlencomp}.  We begin with Proposition~\ref{p:arithstep}, which demonstrates how the value of $d$ can be recovered from $R(\<a, a + d, \ldots, a + kd\>)$.  

\begin{prop}\label{p:arithstep}
Fix an arithmetical numerical monoid $S = \<a, a + d, \ldots, a + kd\>$ with $\gcd(a,d) = 1$ and $1 \le k < a$.  We have 
$$d = \frac{(g-1)(f-1)}{g-f}$$
where $1 < f < g$ are the three minimal values in $R(S)$.  
\end{prop}

\begin{proof}
First, suppose $k = 1$.  The maximal $S$-elasticity tuple $(1,0,2a + d - 2)$ gives $f = \rho_S(1,0,2a + d - 2)$ by Lemma~\ref{l:elastcomp}.  We claim $g = \rho_S(1,0,2a + d - 1)$.  Fix an $S$-elasticity tuple $(c,0,x)$ with $\rho_S(c,0,x) > f$.  If $c = 1$, then by Lemma~\ref{l:elastcomp}, $\rho_S(c,0,x) \ge \rho_S(c,0,2a + d - 1)$.  If $c \ge 2$, then by Lemma~\ref{l:elastcomp}, $\rho_S(c,0,x)$ is minimal when $c = 2$ and when $(c,0,x)$ is maximal, that is, when $x = 2a + d - 2$.  Notice that 
$$\begin{array}{rcl}
(4a + 3d - 2)(3a + d - 1) &=& (4a + d - 2)(3a + d - 1) + d(6a + 2d) - 2d \\
&\ge& (4a + d - 2)(3a + d - 1) + d(4a + d) - 2d \\
&=& (3a + 2d - 1)(4a + d - 2),
\end{array}$$
which means 
$$\rho_S(2,0,2a + d - 2) = \frac{4a + 3d - 2}{4a + d - 2} \ge \frac{3a + 2d - 1}{3a + d - 1} = \rho_S(1,0,2a + d - 1).$$
Subsituting these values for $f$ and $g$ gives 
$$\frac{(g-1)(f-1)}{g-f} = \frac{d^2}{(3a + d - 1)(3a + d - 2)} \cdot \frac{(3a + d - 1)(3a + d - 2)}{d(3a + d - 1) - d(3a + d - 2)} = d,$$
as desired.  

Now, suppose $k \ge 2$, and let $B = \lfloor (3a - 2)/k \rfloor + d$.  We will show that $f = (B + d)/B$ and $g = (B - 1 + d)/(B - 1)$, 
from which the claim follows directly.  Indeed, solving the first equality for $B$ yields $B = d/(f-1)$, and substituting into the second yields 
$$\begin{array}{rcl}
g &=& (B - 1 + d)/(B - 1) \\
&=& (df - (f-1))/(d - (f-1)).
\end{array}$$
Clearing the denominator on the right hand side yields 
$$gd - fd = (g-1)(f-1),$$
and dividing by $g-f$ yields the desired equality.  

By Theorem~\ref{t:arithelastset}, $f = \rho_S(c,s,x)$ for some $S$-elasticity tuple $(c,s,x)$.  
By Lemma~\ref{l:elastcomp}, $(c,s,x)$ is maximal, and since $f > 1$, we have $c = 0$ and $s = 1$.  
This gives the desired form for $f$.  It remains to prove that $g = (B - 1 + d)/(B - 1)$.  

Fix a $S$-elasticity tuple $(c,s,x)$ and let $g' = \rho_S(c,s,x)$.  By Lemma~\ref{l:elastcomp}, it suffices to assume $(c,s,x)$ is maximal.  
If $c = s = 0$, then $g' = 1$, and if $c = 0$ and $s = 1$, then $g' = g$.  First, suppose $k = 2$.  By Lemma~\ref{l:elastcomp}, we can assume 
$c = 1$ and $s = 0$, meaning $x = \lfloor (2a - 2)/2 \rfloor + d = a + d - 1$.  Notice that
$$2B = 2 \lfloor 3a/2 \rfloor + 2d - 2 \ge 3a + 2d - 2 \ge 2a + 2d + 1 \ge 2a + d + 1.$$
Manipulating the above inequality yields 
$$(2a + 3d - 1)(B - 1) \ge (2a + d - 1)(B + d - 1),$$
which gives
$$g' = \frac{c(a + kd) + x + sd}{ca + x} = \frac{2a + 3d - 1}{2a + d - 1} \ge \frac{B + d - 1}{B - 1} = g.$$
Now, suppose $k > 2$.  By Lemma~\ref{l:elastcomp}, it suffices to assume $c = 0$ and $s = 2$, 
and maximality of $x$ gives $x = \lfloor (4a - 2)/k \rfloor + d$.  Notice that
$$\begin{array}{rcl}
2B - d - 2 &=&  2\lfloor (3a - 2)/k \rfloor + d - 2 \ge 2 \lfloor (3a - 2)/k \rfloor - 1 \ge \lfloor (6a - 4)/k \rfloor - 2 \\
&\ge& \lfloor (4a - 2)/k \rfloor + 2 \lfloor (a-1)/k \rfloor - 2 \ge \lfloor (4a - 2)/k \rfloor = x - d.
\end{array}$$
Manipulating the inequality above yields 
$$(x + 2d)(B - 1) \ge x(B - 1 + d)$$
which gives
$$\frac{x + 2d}{x} \ge \frac{B - 1 + d}{B - 1} = g.$$
This completes the proof.  
\end{proof}

\begin{remark}\label{r:arithstep}
Fix an arithmetical numerical monoid $S = \<a, a + d, \ldots, a + kd\>$.  Since $\sup R(S) = (a + kd)/a = 1 + d(k/a)$ by Theorem~\ref{t:holdenmoore}, Proposition~\ref{p:arithstep} also implies that we can recover $a/k$ from $R(S)$.  
\end{remark}

\begin{example}\label{e:extraelast}
The majority of the proof of Theorem~\ref{t:elastlencomp} considers two arithmetical numerical monoids $S = \<a, a + d, \ldots, a + kd\>$ and $S' = \<a', a' + d', \ldots, a' + k'd'\>$ satisfying $d = d'$, $\frac{a}{k} = \frac{a'}{k'}$, $\gcd(a',k') = 1$, and $\gcd(a,k) \ge 2$.  In this case, the sets $R(S')$ and $R(S)$ are nearly identical, as are the elasticities achieved within their respective ``slices''.  Figure~\ref{f:arith_missing_elast} plots the elasticities of $S = \<14, 17, \ldots, 32\>$ and $S' = \<7, 10, 13, 16\>$, and the red points mark the (sparse) elasticities in $R(S) \setminus R(S')$.  In general, every elasticity that lies in $R(S) \setminus R(S')$ is achieved by a maximal $S$-elasticity tuple.  Lemma~\ref{l:extraelast} produces an $(c,s,x) \in \mathcal E(S)$ such that $\rho_S(c,s,x) \notin R(S')$, and the proof of Theorem~\ref{t:elastlencomp} verifies that this is the case.  
\end{example}

\begin{figure}
\begin{center}
\includegraphics[width=2.9in]{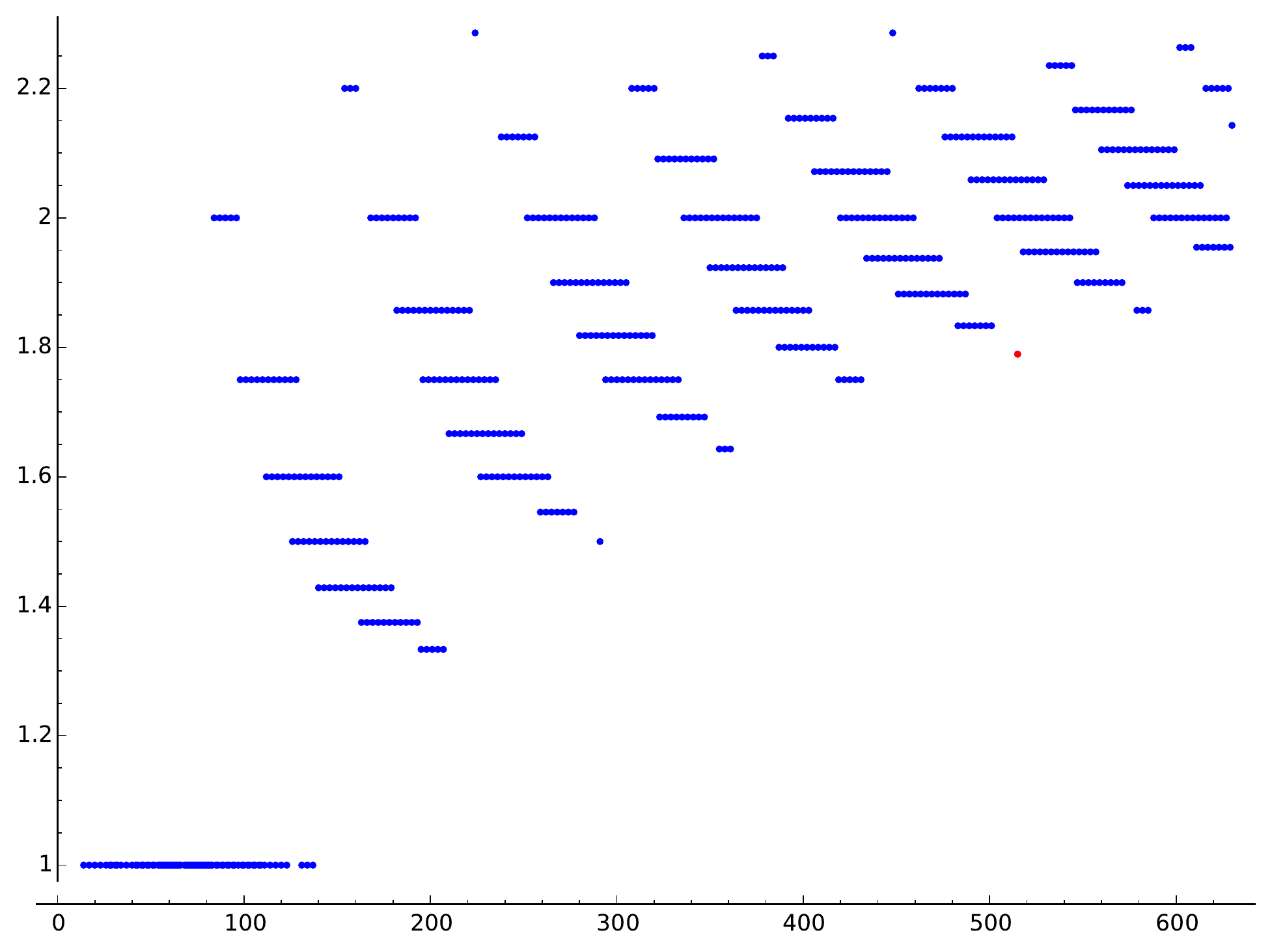}
\hspace{0.2in}
\includegraphics[width=2.9in]{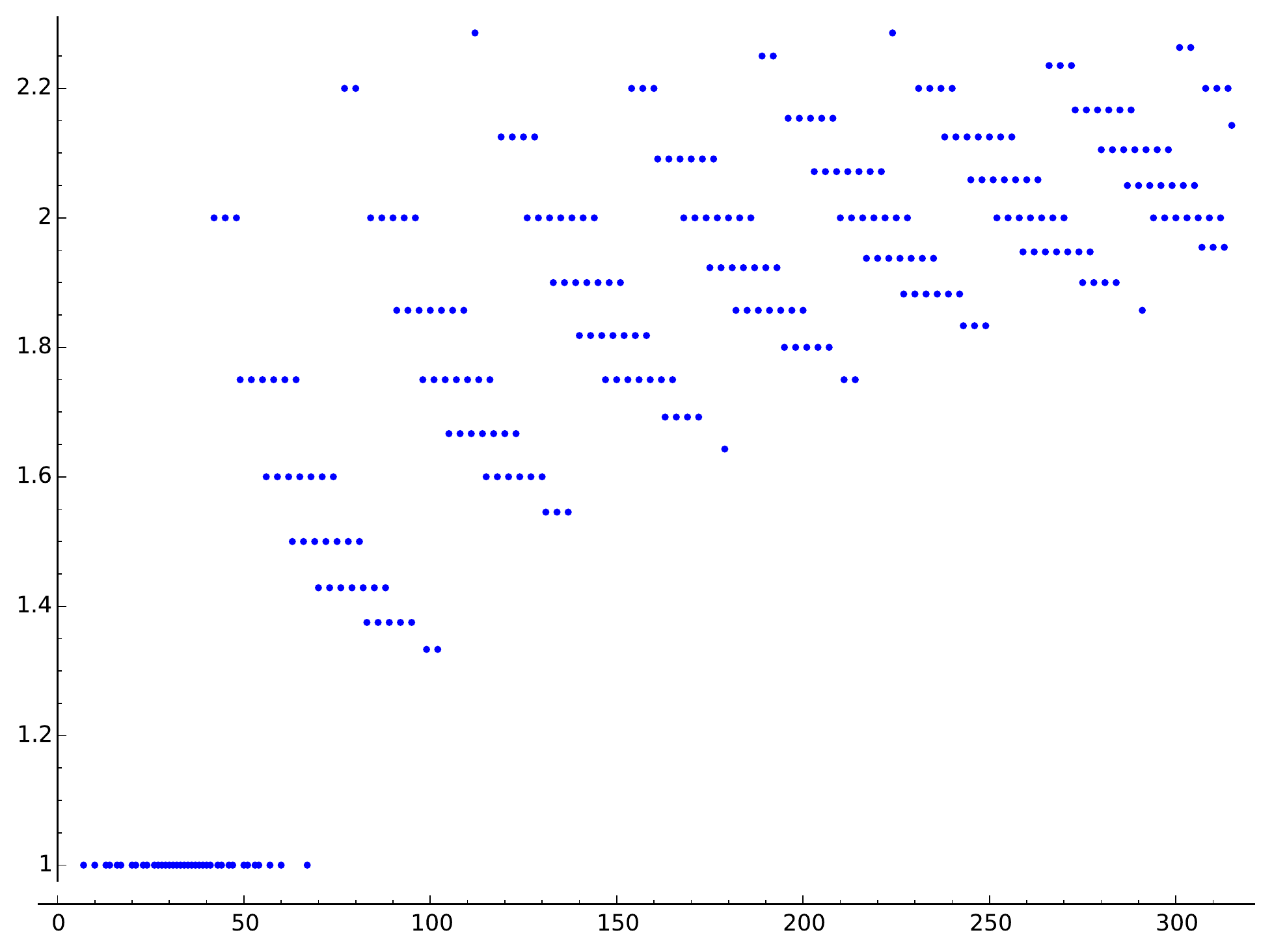}
\end{center}
\caption{Plot of elasticities of $S = \<14,17,20,23,26,29,32\>$ (left) and $S' = \<7, 10, 13, 16\>$ (right).}
\label{f:arith_missing_elast}
\end{figure}

\begin{lemma}\label{l:extraelast}
Fix an arithemtical numerical monoid $S = \<a, a + d, \ldots, a + kd\>$ with $\gcd(a,d) = 1$ and $1 \le k < a$, and suppose $\gcd(a,k) \ge 2$.  
Write $g = \gcd(a,k)$, $k' = k/g$, and $a' = a/g$.  There exists a maximal $S$-elasticity tuple $(c,s,x)$ such that
\begin{enumerate}[(a)]
\item $a'(s+2) \equiv 1 \bmod k'$, and
\item $\gcd(ca + x, ck + s) = 1$.  
\end{enumerate}
\end{lemma}

\begin{proof}
Let $s$ denote the integer satisfying $0 \le s < k'$ and $a'(s+2) \equiv 1 \bmod k'$, and let
$$x = \left\lfloor \frac{(s+2)a - 2}{k} \right\rfloor + d = \left\lfloor \frac{(s+2)a' - (2/g)}{k'} \right\rfloor + d = \frac{(s+2)a' - 1}{k'} + d$$
denote the value such that $(0,s,x)$ is a maximal $S$-elasticity tuple.  Since $\gcd(a',k') = 1$, there exist integers $p$ and $q$ such that $pa' + qk' = 1$.  
Notice that 
$$xk' = (s+2)a' - 1 + k'd > sa',$$
so fix $b \in \ZZ$ such that $b(sa' - xk') > px + qs$.  For $m = 1 - (p + bk')x - (q - ba')s > 0$, we have 
$$(p + bk')(ma' + x) + (q - ba')(bk' + s) = m + (p + bk')x + (q - ba')s = 1,$$
meaning $\gcd(ma' + x, mk' + s) = 1$.  Write $m = cg + r$ for $0 \le r < g$.  We see that the $S$-elasticity tuple $(c, s + rk', x + ra')$ is maximal since 
$$x + ra' = \frac{(s+2)a' - 1}{k'} + d + ra' = \frac{(s+rk'+2)a' - 1}{k'} + d,$$
and $\gcd(ca + (x + ra'), ck + (s + rk')) = \gcd(ma' + x, mk' + s) = 1$, as desired.  
\end{proof}

\begin{thm}\label{t:elastlencomp}
If $S = \<a, a + d, \ldots, a + kd\>$ and $S' = \<a', a' + d', \ldots, a' + k'd'\>$ are distinct arithmetical numerical monoids, then $R(S) = R(S')$ if and only if 
$d = d'$, $\frac{a}{k} = \frac{a'}{k'}$, $\gcd(a,k) \ge 2$ and $\gcd(a',k') \ge 2$.  
\end{thm}

\begin{proof}
If the given conditions are satisfied, then Theorem~\ref{t:setoflengthsets} implies $\mathcal L(S) = \mathcal L(S')$ and thus $R(S) = R(S')$.  Conversely, if $R(S) = R(S')$, then $d = d'$ by Proposition~\ref{p:arithstep} and $\frac{a}{k} = \frac{a'}{k'}$ by Remark~\ref{r:arithstep}.  To complete the proof, it suffices to show that if $d = d'$, $\frac{a}{k} = \frac{a'}{k'}$, $\gcd(a',k') = 1$ and $\gcd(a,k) \ge 2$, then $R(S) \supsetneq R(S')$.  

Since $\frac{a}{k} = \frac{a'}{k'}$, we have $a = ga'$ and $k = gk'$ for $g = \gcd(a,k)$.  Define a map $\phi:\mathcal E(S') \to \mathcal E(S)$ given by $(c',s',x') \mapsto (q,s' + rk',x' + ra')$, where $c' = qg + r$ for $0 \le r < g$.  Notice that $(q,s' + rk',x' + ra') \in \mathcal E(S)$ since $0 \le s' + rk' < k' + (g-1)k' = k$, 
\begin{equation}\label{eq:bounds1}
\left\lceil \frac{(s' + rk')a}{k} \right\rceil = \left\lceil \frac{s'a'}{k'} \right\rceil + ra' \le x' + ra',
\end{equation}
and 
\begin{equation}\label{eq:bounds2}
\begin{array}{rcl}
x' + ra' &\le& \left\lfloor \frac{(s' + 2)a' - 2)}{k'} \right\rfloor + d + ra' = \left\lfloor \frac{(s' + rk' + 2)a - 2g)}{k} \right\rfloor + d \\
&\le& \! \left\lfloor \frac{(s' + rk' + 2)a - 2}{k} \right\rfloor + d.
\end{array}
\end{equation}
We also have $\rho_{S'}(c',s',x') = \rho_S(q,s' + rk',x' + ra')$, so $\phi$ preserves elasticity values.  

Now, by Lemma~\ref{l:extraelast}, there exists a maximal $S$-elasticity tuple $(c,s,x)$ satisfying $a'(s+2) \equiv 1 \bmod k'$ and $\gcd(ca + x, ck + s) = 1$.  If $(c,s,x) \in \image(\phi)$, then it is the image of $(cg + r,s',x - ra') \in \mathcal E(S')$.  In particular, since
$$\begin{array}{rcl}
x - ra' &=& \left\lfloor \frac{(s + 2)a - 2}{k} \right\rfloor + d - ra' = \left\lfloor \frac{(s + 2)a' - (2/g)}{k'} \right\rfloor + d - ra' \\
&=& \! 1 + \left\lfloor \frac{(s + 2)a' - 2}{k'} \right\rfloor + d - ra' = 1 + \left\lfloor \frac{(s'+2)a' - 2}{k'} \right\rfloor + d,
\end{array}$$
we must have $(c,s,x) \notin \image(\phi)$.  Moreover, for $(c_0, s_0, x_0) \in \mathcal E(S)$, if $c_0k + s_0 > ck + s$, then $\rho_S(c,s,x) < \rho_S(c_0,s_0,x_0)$ by Lemma~\ref{l:elastcomp},  and if $c_0k + s_0 < ck + s$, then
$$c_0a + x_0 \ne \frac{(ca + x)(c_0k + s_0)}{ck + s}$$
as the right hand side is not an integer.  This means 
$$\rho_S(c,s,x) = \frac{ca + x + d(ck + s)}{ca + x} \ne \frac{c_0a + x_0 + d(c_0k + s_0)}{c_0a + x_0} = \rho_S(c_0,s_0,x_0).$$
We conclude that the elasticity $\rho_S(c,s,x) \in R(S)$ is only achieved by $(c,s,x)$, which implies $\rho_S(c,s,x) \notin R(S')$ and completes the proof.  
\end{proof}

\begin{cor}\label{c:elastlencomp}
Two arithmetical numerical monoids $S$ and $S'$ satisfy $\mathcal L(S) = \mathcal L(S')$ if and only if $R(S) = R(S')$.  
\end{cor}

We conclude this section with Example~\ref{e:elastlencomp}, which shows the ``arithmetical'' hypothesis in Corollary~\ref{c:elastlencomp} cannot be omitted. 

\begin{example}\label{e:elastlencomp}
Corollary~\ref{c:elastlencomp} shows that for an arithmetical numerical monoid $S$, computation of the elasticity set $R(S)$ (which is given in Theorem~\ref{t:arithelastset}) is just as useful as computing the entire set of length sets $\mathcal L(S)$.  This need not be true in general.  Let $S =  \<6, 10, 13, 14\>$ and $S' =  \<6, 11, 13, 14\>$.  A simple computation shows that 
$$\{4,6\} \in \mathcal L(S) \setminus \mathcal L(S') \text{ and } \rho_S(S \cap [1, 266]) = \rho_{S'}(S' \cap [1, 266]),$$
after which Theorems~\ref{t:maxquasi} and ~\ref{t:minquasi} guarantee that $R(S) = R(S')$.  
\end{example}

\begin{remark}\label{r:elastlencomp}
It remains an interesting question to characterize which numerical monoids $S$ and $S'$ satisfty $R(S) = R(S)$ and $\mathcal L(S) \ne \mathcal L(S')$.  Investigating this phenomenon for general numerical monoids -- or even for specific classes such as those with three minimal generators -- would be of much interest.  
\end{remark}

%%%%%%%%%%%%%%%%%%%%%%%%%%%%%%%%%%%%%%%%%%%%%%%%%%%%%%%%%%%%%%%%%%%%%%%%%%%%%%%%%%%%%%%%%%%%%%%%%%%%%%
\section{The Set of Elasticities for General Numerical Monoids}\label{s:quasi}%%%%%%%%%%%%%%%%%%%%%%%%
%raggedbottom%%%%%%%%%%%%%%%%%%%%%%%%%%%%%%%%%%%%%%%%%%%%%%%%%%%%%%%%%%%%%%%%%%%%%%%%%%%%%%%%%%%%%%%%%

While Theorem~\ref{t:holdenmoore} provides a concise description of the maximal elasticity attained in a numerical monoid $S$ and a coarse topological property of the set of elasticities of $S$, it does not give a full description of $R(S)$.  In this section, we provide such a description by showing that the functions $M(n)$ and $m(n)$ enjoy a powerful linearity~property.  

We begin with a combinatorial lemma.  

\begin{lemma}\label{l:subcol}
Let $k \ge 0$, and fix $c_1, c_2, \ldots, c_r \in \ZZ$ with $r \ge k$.  There exists $T \subsetneq \{1, \ldots, r\}$ 
satisfying $\sum_{i \in T} c_i \equiv \sum_{i = 1}^r c_i \bmod k$.  
\end{lemma}

\begin{proof}
Let $s_j = \sum_{n = 1}^j c_n$ for $j \in \{0, \ldots, r\}$.   The sequence $s_0, s_1, \cdots, s_r$ has length $r + 1 > k$, 
so by the pigeonhole principle, $s_i \equiv s_j \bmod k$ for some $i < j$.   This means $s_j - s_i \equiv 0 \bmod k$, 
so choosing $T = \{1, \ldots, r\} \setminus \{i + 1, \ldots, j\}$ completes the proof.  
\end{proof}

\begin{thm}\label{t:maxquasi}
Given a numerical monoid $S = \<g_1, \ldots, g_k\>$ minimally generated by $g_1 < \cdots < g_k$, the maximal factorization length function $M: S \to \NN$ satisfies 
$$M(n) = M(n - g_1) + 1$$ 
for all $n > (g_1 - 1)g_k$.  
\end{thm}

\begin{proof}
Fix a factorization $\vec a$ for $n$, and suppose that $a_2 + \cdots + a_k \ge g_1$.  
Since $a_1g_1 + a_2g_2 + \cdots + a_kg_k = n$, we have $a_2g_2 + \cdots + a_kg_k \equiv n \bmod g_1$.  
Viewing this sum as $a_2 + \cdots + a_k$ integers selected from $\{g_2, \ldots, g_k\}$, 
Lemma~\ref{l:subcol} guarantees the existence of $b_2, \ldots, b_k \ge 0$ such that 
(i)  $b_i \le a_i$ for each $i > 1$, 
(ii) $\sum_{i = 2}^k a_i > \sum_{i = 2}^k b_i$, and 
(iii) $b_2g_2 + \cdots + b_kg_k \equiv n \bmod g_1$.  
This implies $b_2g_2 + \cdots + b_kg_k < a_2g_2 + \cdots + a_kg_k$, so there exists $b_1 \ge 0$ so that $\vec b = (b_1, b_2, \ldots, b_k) \in \mathsf Z(n)$.  This gives 
$$(b_1 - a_1)g_1 = \sum_{i = 2}^k (a_i - b_i)g_i > \sum_{i = 2}^k (a_i - b_i)g_1,$$
from which canceling $g_1$ yields $|\vec b| > |\vec a|$.  

Now, suppose that $\vec a \in \mathsf Z(n)$ is maximal.  The above argument implies that $a_2 + \cdots + a_k < g_1$.  
In particular, if $n > (g_1 - 1)g_k$, we must have $a_1 > 0$.  
This means $\vec a - \vec e_1 \in \mathsf Z(n - g_1)$, so we have $M(n - g_1) \ge |\vec a| - 1$, 
and since $\vec a$ has maximal length, we have $M(n - g_1) = |\vec a| - 1$.  This completes the proof.  
\end{proof}

The proof of the following analogous result is almost identical to the proof of Theorem~\ref{t:maxquasi} and hence omitted.  

\begin{thm}\label{t:minquasi}
Given a numerical monoid $S = \<g_1, \ldots, g_k\>$ minimally generated by $g_1 < \cdots < g_k$, the minimal factorization length function $m: S \to \NN$ satisfies
$$m(n) = m(n - g_k) + 1$$
for all $n > (g_k - 1)g_{k-1}$.  
\end{thm}

\begin{figure}
\begin{center}
\includegraphics[width=2.9in]{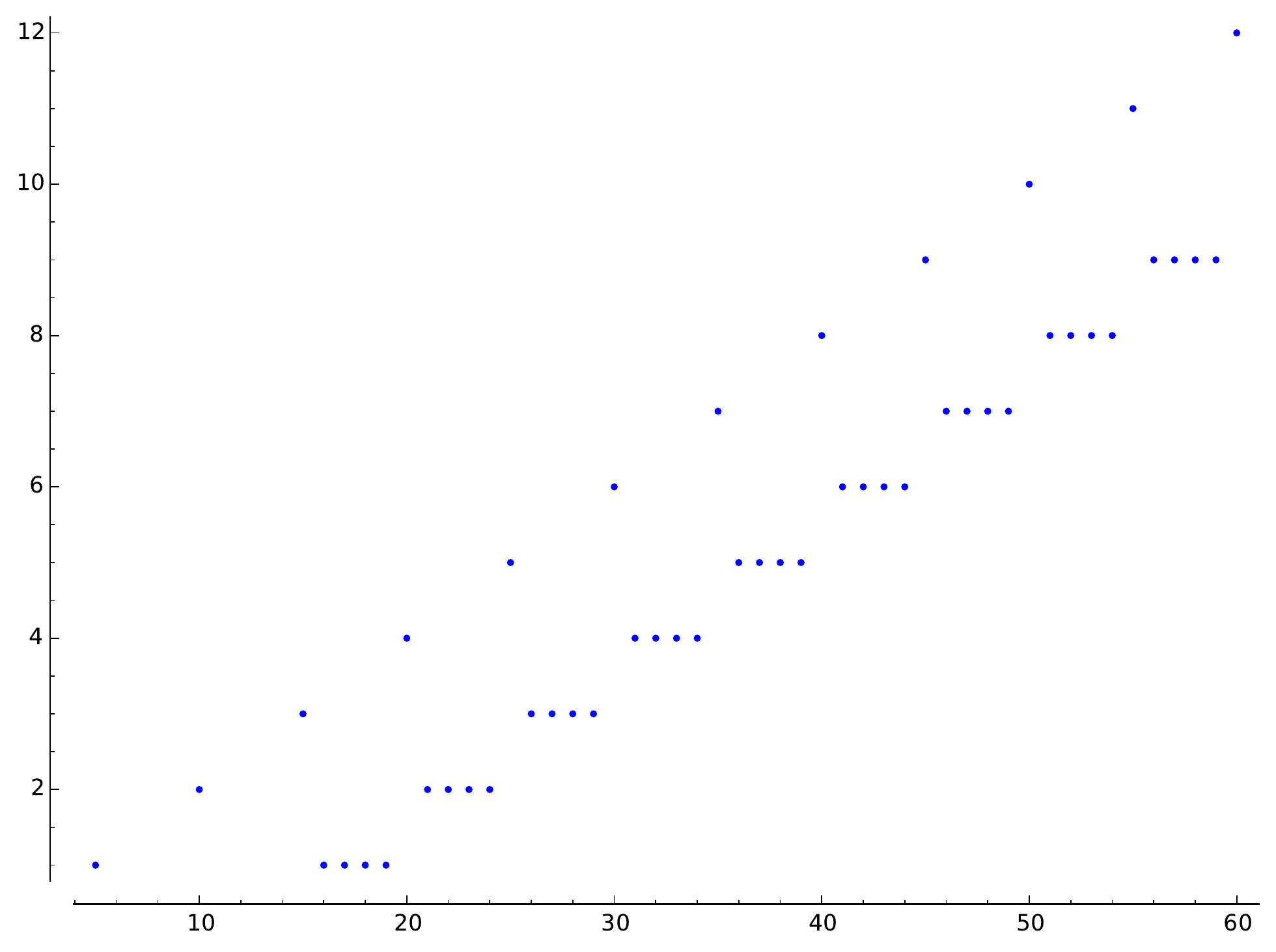}
\hspace{0.2in}
\includegraphics[width=2.9in]{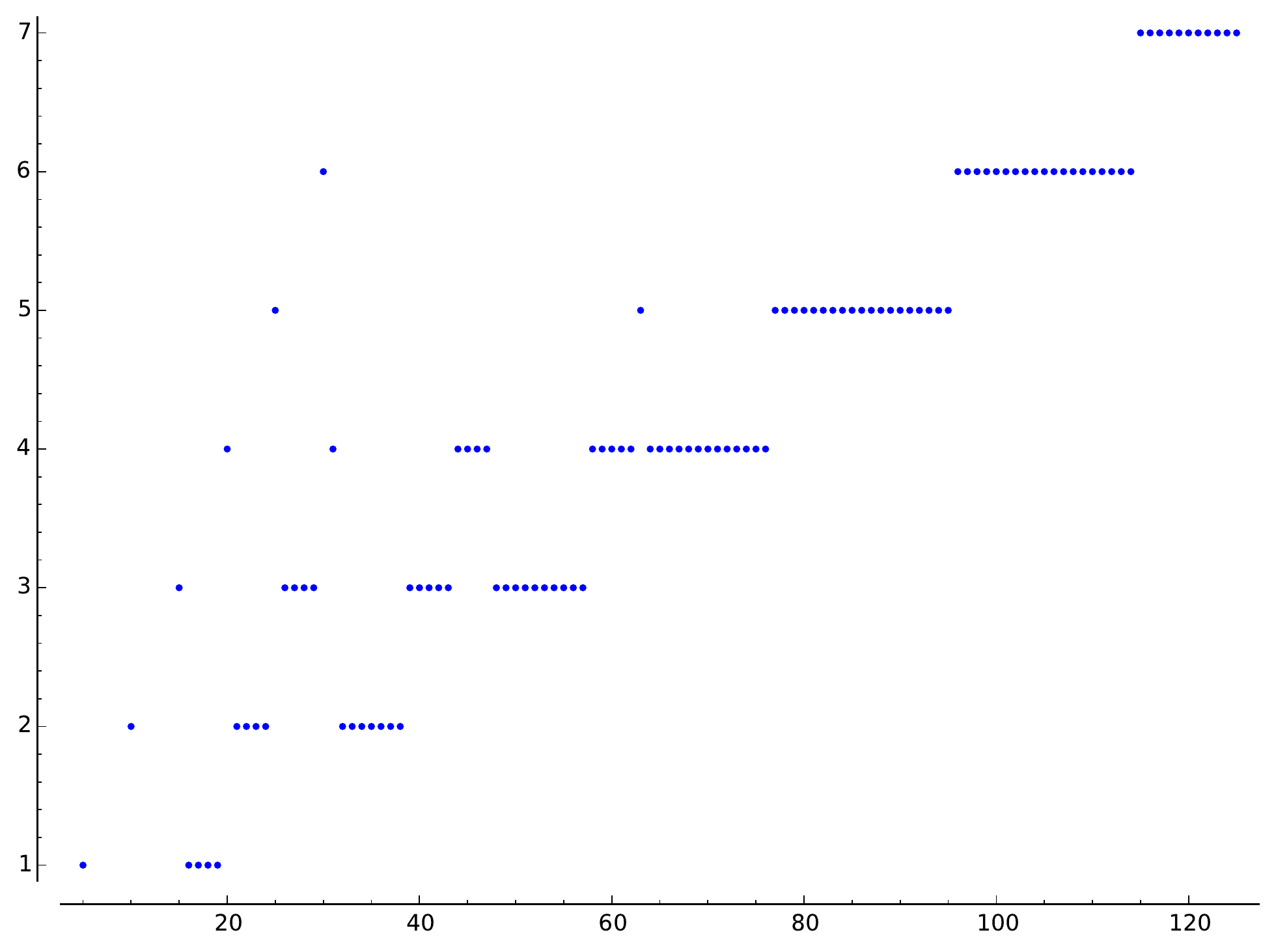}
\end{center}
\caption{Plots of $M(n)$ (left) and $m(n)$ (right) for the numerical monoid $\<5,16,17,18,19\>$.}
\label{f:fact_plots}
\end{figure}

\begin{example}\label{e:factorizationgraphs}
If a numerical monoid $S$ has $g_1$ as its smallest minimal generator then Theorems~\ref{t:maxquasi} and~\ref{t:minquasi} state that $M: S \to \NN$ as a function will eventually manifest graphically as a collection of $g_1$ discrete lines with a common slope of $1/g_1$.  Similarly, if $g_k$ is the largest minimal generator, then the graph of $m: S \to \NN$ will eventually appear as a collection of $g_k$ discrete lines with common slope $1/g_k$.  Figure~\ref{f:fact_plots}, which shows the functions $M(n)$ and $m(n)$ for $S = \<5,16,17,18,19\>$, demonstrates this concept.
\end{example}

Since the elasticity of an element $n$ in a numerical monoid is given by the quotient of $M(n)$ and $m(n)$, we use Theorems~\ref{t:maxquasi} and~\ref{t:minquasi} to provide a characterization of $R(S)$.

\begin{cor}\label{c:elasticityset}
Fix a numerical monoid $S$ minimally generated by $g_1 < \cdots < g_k$.  
\begin{enumerate}[(a)]
\item For $n \ge g_{k-1}g_k$, we have 
$$\rho(n + g_1g_k) = \frac{M(n) + g_k}{m(n) + g_1}.$$
\item The set $R(S)$ is the union of a finite set and a collection of $g_1g_k$ monotone increasing sequences, each converging to $g_k/g_1$.  
\end{enumerate}
\end{cor}

\begin{proof}
Part~(a) follows directly from Theorems~\ref{t:maxquasi} and~\ref{t:minquasi}.  
From this, it follows that for $g_{k-1}g_k \le n < g_{k-1}g_k + g_1g_k$, the sequence 
$$\rho(n), \rho(n + g_1g_k), \rho(n + 2g_1g_k), \ldots$$
is monotone increasing and converges to $g_k/g_1$.  This completes the proof.  
\end{proof}

\begin{remark}\label{r:elasticities}
Theorem~\ref{t:holdenmoore} states that the only accumulation point of the elasticity set $R(S)$ is its maximum.  Corollary~\ref{c:elasticityset}, on the other hand, gives a characterization of the entire set $R(S)$, from which several other results from~\cite{elasticity} can be recovered.   In particular, the characterization of the set of elasticities provided in Corollary~\ref{c:elasticityset} describes $R(S)$ as a union of a finite set and $g_1g_k$ monotone increasing sequences, each converging to $g_k/g_1$, which clearly implies that the only accumulation point is~$g_k/g_1$.
\end{remark}

\begin{figure}
\begin{center}
\includegraphics[width=2.9in]{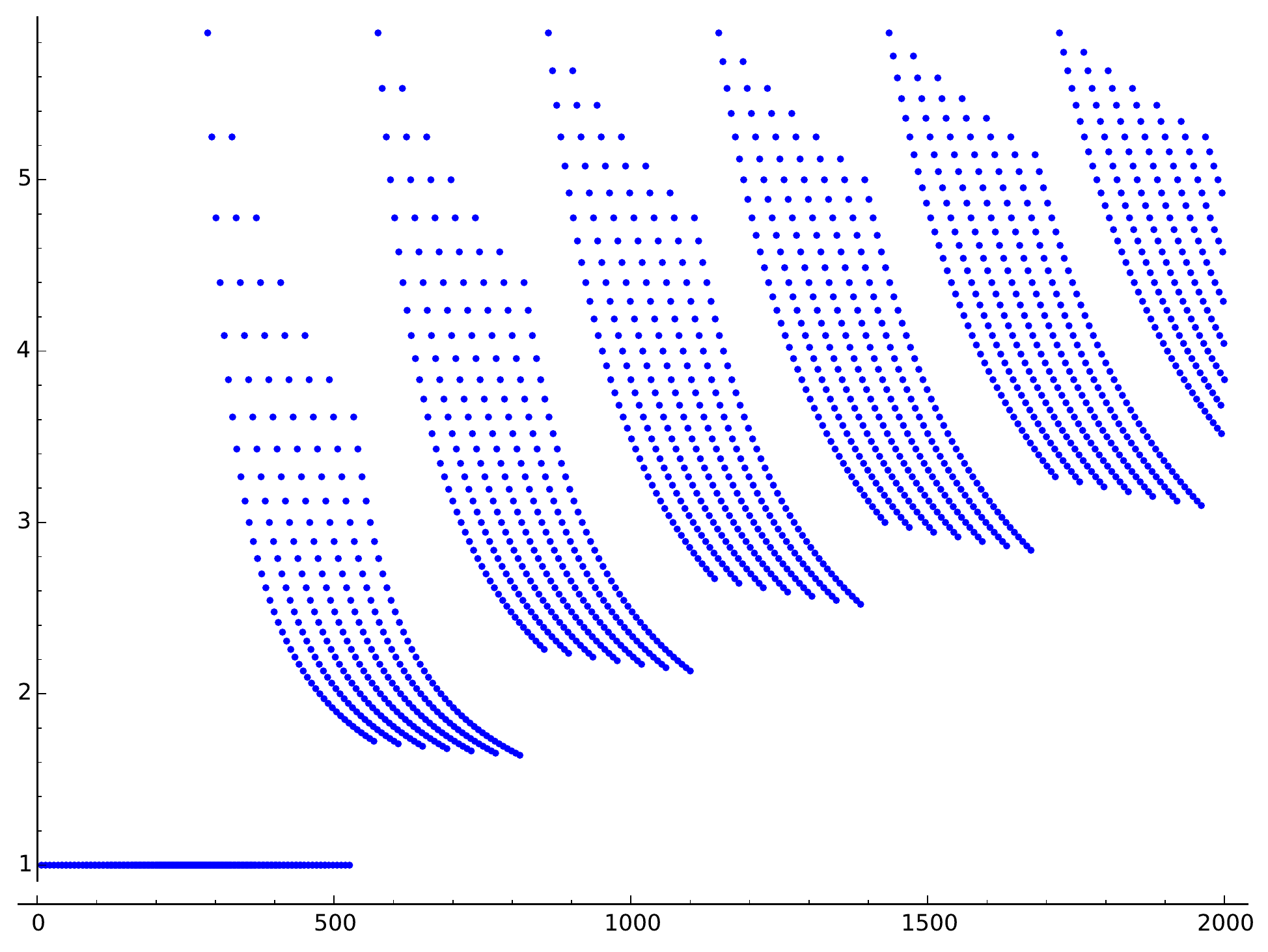}
\hspace{0.2in}
\includegraphics[width=2.9in]{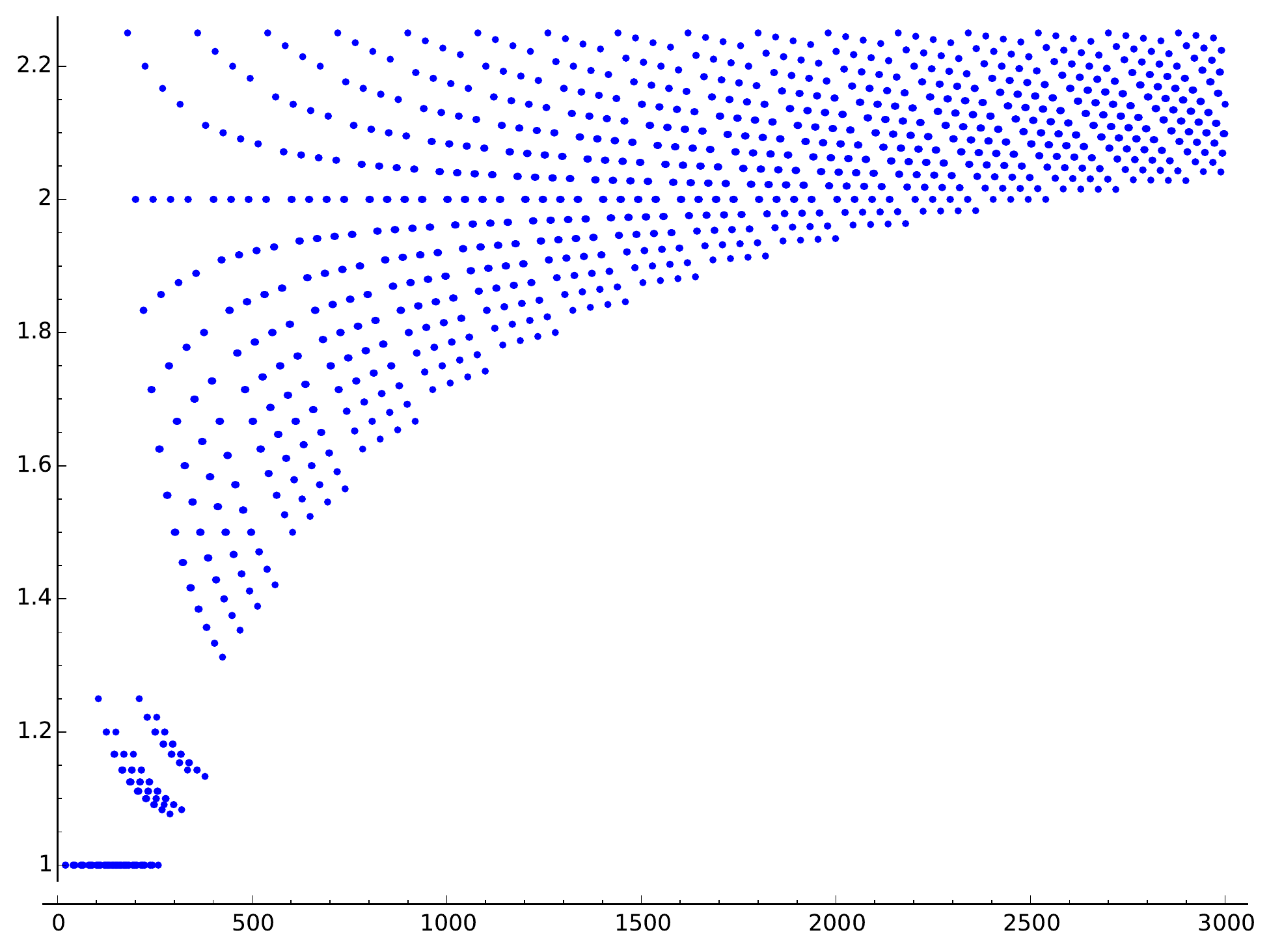}
\end{center}
\caption{Plots of the elasticities of elements from the monoids $\<7,41\>$ (left) and $\<20,21,45\>$ (right).}
\label{f:elast_plots}
\end{figure}

\begin{example}\label{e:rho_plots}

The elasticity graphs for numerical monoids $\<7,41\>$ and $\<20,21,45\>$ are given in Figure~\ref{f:elast_plots}.  The latter of numerical monoids is not arithmetical and demonstrates that the uniformity of the ``slices'' enjoyed by arithmetical numerical monoids is not present, especially for smaller values.  Regardless, for any numerical monoid $S = \<g_1, \ldots, g_k\>$, the characterization of $R(S)$ provided in Corollary~\ref{c:elasticityset} shows that $\rho$ can be eventually described as $g_1g_k$ monotone increasing sequences that limit to $g_k/g_1$, where each sequence contains precisely one point in each ``slice.''
  
\end{example}

\section{Acknowledgements}

The authors would like to thank Dr. Scott Chapman, Alfred Geroldinger, and Sherilyn Tamagawa for various helpful conversations and insights.  The third author is funded by National Science Foundation grant DMS-1045147.

%%%%%%%%%%%%%%%%%%%%%%%%%%%%%%%%%%%%%%%%%%%%%%%%%%%%%%%%%%%%%%%%%%%%%%%%%
%%%%%%%%%%%%%%%%%%%%%%%%%%%%%%%%%%%%%%%%%%%%%%%%%%%%
%%%%%%%%%%%%%%%%%%%%%%%%%%%%%%%%%%%%%%%%%%%%%%%%%%%%%%%%%%%%%%%%%%%%%%%%%

%%%%%%%%%%%%%%%%%%%%%%%%%%%%%%%%%%%%%%%%%%%%%%%%%%%%%%%%%%%%%%%%%%%%%%%%%
\end{document}